\documentclass[11pt,reqno]{amsart}

\usepackage[a4paper,margin=2.7cm]{geometry}
\usepackage{amsmath,amssymb,amsthm}
\usepackage{booktabs}
\usepackage{microtype}
\usepackage{xcolor}
\usepackage[colorlinks=true,linkcolor=blue!55!black,citecolor=blue!55!black,urlcolor=blue!55!black]{hyperref}

\theoremstyle{plain}
\newtheorem{theorem}{Theorem}[section]
\newtheorem{proposition}[theorem]{Proposition}
\newtheorem{lemma}[theorem]{Lemma}
\newtheorem{corollary}[theorem]{Corollary}
\theoremstyle{definition}
\newtheorem{definition}[theorem]{Definition}
\newtheorem{example}[theorem]{Example}
\newtheorem{remark}[theorem]{Remark}
\newtheorem{observation}[theorem]{Computational Observation}
\newtheorem{problem}{Problem}

\DeclareMathOperator{\Stab}{Stab}
\DeclareMathOperator{\Aut}{Aut}
\newcommand{\rM}{r_{M}}
\newcommand{\rW}{r_{W}}
\newcommand{\EM}{E_{M}}
\newcommand{\EW}{E_{W}}

\title[Regular anti-phase templates in the stable marriage problem]
{Regular Anti-Phase Templates in the Stable Marriage Problem:\\
A Generator Criterion, its Converse, and a Counting Bound}

\author{Yoshiteru Ishida}
\date{Version 15, July 2026}


\begin{document}

\begin{abstract}
We study a family of highly symmetric instances of the stable marriage problem built from
regular actions of finite groups. Given a finite group $G$ of order $n$ and an ordering
$A=(a_0,\dots,a_{n-1})$ of its elements, the \emph{regular anti-phase template} $P(G,A)$ is
the profile in which $m_g$ ranks $w_{ga_0}\succ\cdots\succ w_{ga_{n-1}}$ and $w_h$ ranks
$m_{ha_{n-1}^{-1}}\succ\cdots\succ m_{ha_0^{-1}}$. Our first result explains the word
\emph{anti-phase}: among all automorphism-maximal profiles of size $n$, the templates $P(G,A)$
are \emph{exactly} those satisfying the constant rank-sum identity
$\rM(m,w)+\rW(w,m)=n+1$. The construction is therefore canonical rather than ad hoc.

The template has $n$ canonical stable matchings $\mu_t(m_g)=w_{ga_t}$. Writing
$q_b=a_{b-1}a_b^{-1}$ for the adjacent quotients, we prove a coset lemma: for any stable
matching, the top level set of its index function is a union of left cosets of
$\langle q_b\rangle$, where $b$ is the maximal index. This yields a sharpened generator
criterion, and — our main new result — its exact converse:
\[
  |\Stab(P(G,A))| = n
  \iff
  \langle q_b\rangle = G \text{ for every } b=1,\dots,n-1 .
\]
We further give the counting bound
$|\Stab(P(G,A))| \ge n + \sum_{b=1}^{n-1}\bigl(2^{[G:\langle q_b\rangle]}-2\bigr)$,
which is sharp for all groups of order at most $5$ and for both groups of order $4$; in
particular it yields $|\Stab(P(V_4,A))| \ge 4+3(2^2-2) = 10$, with equality confirmed by
enumeration; this gives a structural explanation of the classical maximum $f(4)=10$.

Two consequences correct the picture suggested by earlier versions of this work. If $|G|=p$
is prime, then $P(\mathbb{Z}_p,A)$ has exactly $p$ stable matchings for \emph{every} ordering
$A$; but for composite order the count depends on $A$ even for cyclic groups — the template
$P(\mathbb{Z}_4,(0,2,1,3))$ has $8$ stable matchings, not $4$. The chain-lattice phenomenon
is thus governed by primality, not by cyclicity. Along the way we record the sex-equality
spectrum of the cyclic profile and an odd/even parity theorem, and we note that the Klein
template has two exactly sex-equal stable matchings where the cyclic profile $C_4$ has none.
All computational claims are verified by an accompanying script.
\end{abstract}

\maketitle

\section{Introduction}

The stable marriage problem (SMP), introduced by Gale and Shapley~\cite{GaleShapley},
is one of the basic models of two-sided matching. Its stable matchings form a finite
distributive lattice under the men-dominance order, represented by a rotation
poset~\cite{Gusfield,Manlove}. This paper studies a class of highly symmetric profiles
sitting at the intersection of three themes: Latin preference profiles and constant
rank-sum identities; fairness objectives, especially sex-equality; and automorphism-maximal,
group-theoretically defined templates.

The starting point is the cyclic anti-phase profile $C_n$, given by
\[
  m_i : w_i \succ w_{i+1} \succ \cdots \succ w_{i-1},
  \qquad
  w_j : m_{j+1} \succ m_{j+2} \succ \cdots \succ m_j
\]
with indices modulo $n$. It is a circulant Latin profile, but unlike the Latin instances used
by Benjamin, Converse and Krieger to construct exponentially many stable
matchings~\cite{Benjamin}, it has only $n$ stable matchings, arranged in a single chain. It is
therefore a natural chain-lattice benchmark inside the Latin-marriage family.

\subsection*{What is new in this version}
Earlier versions of this manuscript proved that $C_n$ has exactly $n$ stable matchings,
introduced regular anti-phase templates $P(G,A)$ over an arbitrary finite group, and proved a
sufficient condition — the \emph{generator criterion} — for the canonical matchings to exhaust
the stable set. The present version adds the following.

\begin{enumerate}
\item[(A)] \textbf{A characterization of the templates} (Theorem~\ref{thm:antiphase-char}).
Among all automorphism-maximal size-$n$ profiles, written in the regular normal form
$P(G,A,B)$, the constant rank-sum identity $\rM+\rW\equiv n+1$ holds if and only if
$B=(a_{n-1}^{-1},\dots,a_0^{-1})$, i.e.\ if and only if the profile is an anti-phase template
$P(G,A)$. This removes the appearance of arbitrariness from the definition: the anti-phase
choice is not one convention among many, but the unique one carrying the constant rank sum.

\item[(B)] \textbf{A coset lemma} (Lemma~\ref{lem:coset}) isolating the mechanism behind the
generator criterion: for a stable matching with index function $d$ and maximal index
$b=\max_g d_g$, the level set $\{g : d_g=b\}$ is a union of left cosets of
$\langle q_b\rangle$, where $q_b=a_{b-1}a_b^{-1}$.

\item[(C)] \textbf{An exact converse of the generator criterion}
(Theorem~\ref{thm:converse}), which settles Problem~3 of the previous version's list of open
problems. If some $q_b$ fails to generate $G$, we construct explicitly a family of
non-canonical stable matchings, one for each proper nonempty union of left cosets of
$\langle q_b\rangle$.

\item[(D)] \textbf{A counting bound} (Theorem~\ref{thm:count}), which addresses Problem~4 of
that list:
\[
  |\Stab(P(G,A))| \;\ge\; n + \sum_{b=1}^{n-1}\Bigl(2^{[G:\langle q_b\rangle]}-2\Bigr).
\]
It is sharp for all groups of order $\le 5$ and for both groups of order $4$. For the Klein
four-group template it gives $4+3(2^2-2)=10$, matching the classical maximum $f(4)=10$ and
upgrading the previous version's brute-force observation to a structural lower bound.

\item[(E)] \textbf{A correction.} The previous version suggested that cyclicity is what forces
the chain lattice (``if $G$ is non-cyclic, no single adjacent quotient can generate $G$''). That
is true but incomplete, because it silently fixes the natural ordering $A=(0,1,\dots,n-1)$. The
correct statement is governed by \emph{primality}: for $|G|=p$ prime, every ordering yields
exactly $p$ stable matchings (Corollary~\ref{cor:prime}), whereas for composite order the count
depends on $A$ even for cyclic $G$ — the template $P(\mathbb{Z}_4,(0,2,1,3))$ has $8$ stable
matchings (Example~\ref{ex:z4}).
\end{enumerate}

The automorphism bound $|\Aut(S)|\le n$ (Proposition~\ref{prop:semireg}) is \emph{not} claimed
as new: it is the SMP specialization of the known semiregularity of stabilizers of preference
profiles, appearing for instance in the group-theoretic analysis of symmetric two-sided matching
mechanisms by Bubboloni, Gori and Meo~\cite{Bubboloni}. A short proof is included only to fix
notation.

\subsection*{Organization}
Section~\ref{sec:prelim} fixes definitions. Section~\ref{sec:cyclic} treats the cyclic profile
$C_n$: the Latin and constant rank-sum identities, the shift description of its stable
matchings, its chain lattice, and its sex-equality spectrum with an odd/even parity theorem.
Section~\ref{sec:aut} recalls semiregularity and the regular normal form, and proves the
characterization (A). Section~\ref{sec:templates} develops regular anti-phase templates and
proves (B), (C), (D), (E). Section~\ref{sec:products} records a product convention.
Section~\ref{sec:comp} gives script-backed computations, including a census over orderings.
Section~\ref{sec:open} lists the problems that remain.

\section{Preliminaries}\label{sec:prelim}

Let $M=\{m_0,\dots,m_{n-1}\}$ and $W=\{w_0,\dots,w_{n-1}\}$. Each participant has a strict
total order over all participants on the opposite side; profiles are always \emph{strict and
complete}. A matching is a bijection $\mu:M\to W$. A pair $(m,w)$ \emph{blocks} $\mu$ if $m$
prefers $w$ to $\mu(m)$ and $w$ prefers $m$ to $\mu^{-1}(w)$; $\mu$ is \emph{stable} if no pair
blocks it. The set of stable matchings of a profile $S$ is $\Stab(S)$.

We write $\rM(m,w)$ for the rank of $w$ in $m$'s list (so the top choice has rank $1$) and
$\rW(w,m)$ for the rank of $m$ in $w$'s list. The men-dominance order is $\mu\succeq_M\nu$ if
every man weakly prefers $\mu$; with it, $\Stab(S)$ is a finite distributive lattice
$L(S)$~\cite{Gusfield}. Its rotation poset is $R(S)$, and
$L(S)\cong J(R(S))$, the lattice of order ideals~\cite{Birkhoff,Gusfield}. Set
\[
  \EM(\mu)=\sum_{m\in M}\rM(m,\mu(m)),\qquad
  \EW(\mu)=\sum_{w\in W}\rW(w,\mu^{-1}(w)).
\]
The egalitarian score is $\EM+\EW$~\cite{Irving}; the \emph{sex-equality imbalance} is
$\Delta(\mu)=\EM(\mu)-\EW(\mu)$, and the sex-equal objective minimizes $|\Delta|$ over
$\Stab(S)$~\cite{Kato,McDermid}.

A profile is \emph{rank-Latin} if both rank matrices $(\rM(m_i,w_j))_{i,j}$ and
$(\rW(w_j,m_i))_{j,i}$ are Latin squares.

\section{The circulant anti-phase profile}\label{sec:cyclic}

Throughout this section subscripts are modulo $n$ and $C_n$ is as in the introduction. Thus
$w_{i+d}$ has rank $d+1$ in $m_i$'s list for $d\in\{0,\dots,n-1\}$; in $w_j$'s list, $m_{j+c}$
has rank $c$ for $c=1,\dots,n-1$, and $m_j$ has rank $n$.

\begin{proposition}[Rank-Latin and anti-phase identities]\label{prop:cyclic-latin}
The profile $C_n$ satisfies:
\begin{enumerate}
\item[(a)] both rank matrices are circulant Latin squares;
\item[(b)] $\rM(m_i,w_{i+d})+\rW(w_{i+d},m_i)=n+1$ for every pair;
\item[(c)] $\EM(\mu)+\EW(\mu)=n(n+1)$ for every perfect matching $\mu$.
\end{enumerate}
\end{proposition}

\begin{proof}
We have $\rM(m_i,w_{i+d})=d+1$. For $d=0$, woman $w_i$ ranks $m_i$ last, at rank $n$. For
$d\ge 1$, woman $w_{i+d}$ sees $m_i$ at relative position $-d \pmod n$, hence at rank $n-d$.
The rank sum is $(d+1)+(n-d)=n+1$, valid also at $d=0$. Rows and columns of each rank matrix
are cyclic shifts of $1,\dots,n$, so both are Latin. Summing (b) over the $n$ matched pairs of
any perfect matching gives (c).
\end{proof}

\begin{remark}
Latin stable marriage instances were studied by Benjamin et al.~\cite{Benjamin}, and
Latin-related profiles appear in later enumerative and legal-assignment
contexts~\cite{Borodin,Faenza}. The profile $C_n$ is a minimal chain-lattice benchmark inside
that family.
\end{remark}

For $k=0,\dots,n-1$ define the \emph{shift matching} $\mu_k(m_i)=w_{i+k}$.

\begin{proposition}[Shift description]\label{prop:shift}
For every $n\ge 1$, $\Stab(C_n)=\{\mu_0,\dots,\mu_{n-1}\}$; in particular $|\Stab(C_n)|=n$.
\end{proposition}

\begin{proof}
This is the special case $G=\mathbb{Z}_n$, $A=(0,1,\dots,n-1)$ of Theorem~\ref{thm:gencrit}
below, via Corollary~\ref{cor:cyclic-extreme}. (A direct proof is also easy: it is the argument
of Theorem~\ref{thm:gencrit} written additively.)
\end{proof}

\begin{corollary}\label{cor:chain}
The stable lattice $L(C_n)$ is the $n$-element chain
$\mu_0\succ_M\mu_1\succ_M\cdots\succ_M\mu_{n-1}$, and the rotation poset $R(C_n)$ is an
$(n-1)$-element chain.
\end{corollary}

\begin{proof}
Under $\mu_k$ every man receives rank $k+1$, so the men-dominance order is the displayed chain.
The rotation statement follows from $L(C_n)\cong J(R(C_n))$: the chain with $n$ elements is the
ideal lattice of the chain with $n-1$ elements.
\end{proof}

\begin{corollary}
For $n\ge 2$, $L(C_n)$ is directly indecomposable as a nontrivial product of finite
distributive lattices.
\end{corollary}

\begin{proof}
A nontrivial product $L_1\times L_2$ contains the incomparable elements $(a_1,b_0)$ and
$(a_0,b_1)$ whenever $a_0<a_1$ in $L_1$ and $b_0<b_1$ in $L_2$; hence it is not a chain.
\end{proof}

\begin{proposition}[Rank-energy spectrum]\label{prop:spectrum}
For $k=0,\dots,n-1$,
\[
  \EM(\mu_k)=n(k+1),\qquad \EW(\mu_k)=n(n-k),\qquad \Delta(\mu_k)=n(2k+1-n).
\]
The imbalances form the arithmetic progression
$-n(n-1),\,-n(n-3),\,\dots,\,n(n-3),\,n(n-1)$ with common difference $2n$.
\end{proposition}

\begin{proof}
Under $\mu_k$ every man receives rank $k+1$, so $\EM(\mu_k)=n(k+1)$; by
Proposition~\ref{prop:cyclic-latin}(c), $\EW(\mu_k)=n(n+1)-n(k+1)=n(n-k)$, and the imbalance
formula follows.
\end{proof}

\begin{theorem}[Odd/even sex-equality theorem]\label{thm:parity}
On $\Stab(C_n)$:
\begin{enumerate}
\item[(a)] if $n$ is odd, there is a unique exactly sex-equal stable matching, namely
$\mu_{(n-1)/2}$;
\item[(b)] if $n$ is even, there is no exactly sex-equal stable matching, and the two closest
ones are $\mu_{n/2-1}$ and $\mu_{n/2}$, with imbalances $-n$ and $+n$.
\end{enumerate}
\end{theorem}

\begin{proof}
Exact sex-equality means $n(2k+1-n)=0$, i.e.\ $2k+1=n$. The left-hand side is odd, so a
solution exists precisely when $n$ is odd, and then $k=(n-1)/2$ is unique. If $n$ is even, the
nearest odd integers to $n$ are $n\mp1$, giving $k=n/2-1$ and $k=n/2$, with imbalances
$n(2(n/2-1)+1-n)=-n$ and $n(2(n/2)+1-n)=+n$.
\end{proof}

Theorem~\ref{thm:parity}(b) is a genuine obstruction for the \emph{cyclic} profile only; see
Observation~\ref{obs:klein}, where a non-cyclic template of the same size $n=4$ does possess
exactly sex-equal stable matchings.

\section{Automorphisms, regular normal form, and what ``anti-phase'' means}\label{sec:aut}

An \emph{automorphism} of a profile $S$ is a pair $(\sigma,\tau)$ of permutations of $M$ and
$W$ preserving all ranks:
\[
  \rM(m_i,w_j)=\rM(m_{\sigma(i)},w_{\tau(j)}),\qquad
  \rW(w_j,m_i)=\rW(w_{\tau(j)},m_{\sigma(i)}).
\]

\begin{lemma}[Rigidity]\label{lem:rigid}
If an automorphism of a strict complete profile fixes one man, it is the identity.
\end{lemma}

\begin{proof}
Suppose $(\sigma,\tau)$ fixes $m_i$. Preservation of $m_i$'s strict complete list forces
$\tau$ to fix every woman. Preservation of the women's strict complete lists then forces
$\sigma$ to fix every man.
\end{proof}

\begin{proposition}[Semiregularity bound]\label{prop:semireg}
For every strict complete size-$n$ profile $S$, $\Aut(S)$ acts freely on each side. Hence
$|\Aut(S)|$ divides $n$; in particular $|\Aut(S)|\le n$.
\end{proposition}

\begin{proof}
Freeness on men is Lemma~\ref{lem:rigid}. Every orbit on men therefore has size $|\Aut(S)|$,
so $|\Aut(S)|$ divides $n$. The argument on women is identical.
\end{proof}

\begin{remark}
This is a profile-level specialization of the semiregularity of stabilizers studied in
group-theoretic work on symmetric two-sided matching mechanisms; see, e.g.,
Bubboloni et al.~\cite{Bubboloni}. It is included only for self-contained notation.
\end{remark}

\begin{definition}
A strict complete size-$n$ profile is \emph{automorphism-maximal} if $|\Aut(S)|=n$.
\end{definition}

\begin{theorem}[Regular normal form]\label{thm:normalform}
A strict complete size-$n$ profile is automorphism-maximal if and only if it is isomorphic to a
profile of the following shape. There are a group $G$ of order $n$ and two orderings
$A=(a_0,\dots,a_{n-1})$ and $B=(b_0,\dots,b_{n-1})$ of the elements of $G$ such that, indexing
men and women by $G$,
\[
  m_g:\; w_{ga_0}\succ w_{ga_1}\succ\cdots\succ w_{ga_{n-1}},
  \qquad
  w_h:\; m_{hb_0}\succ m_{hb_1}\succ\cdots\succ m_{hb_{n-1}} .
\]
We denote this profile $P(G,A,B)$. The group $G$ acts by left translation, and this action is
the full automorphism group.
\end{theorem}

\begin{proof}
If $|\Aut(S)|=n$ then, by Proposition~\ref{prop:semireg}, the action on men is free and
transitive. Identify $G=\Aut(S)$ with $M$ by choosing a base man $m_e$ and setting
$m_g=g(m_e)$, and likewise identify $W$ with $G$ via a base woman $w_e$. The base man's list has
the form $w_{a_0}\succ\cdots\succ w_{a_{n-1}}$; applying $g$ gives
$m_g: w_{ga_0}\succ\cdots\succ w_{ga_{n-1}}$, and the women's lists are analogous. Conversely,
any profile of the displayed shape is preserved by every left translation of $G$, so it has at
least $n$ automorphisms; Proposition~\ref{prop:semireg} forces equality.
\end{proof}

\begin{corollary}\label{cor:autmax-latin}
Every automorphism-maximal strict complete profile is rank-Latin.
\end{corollary}

\begin{proof}
In the normal form, for each fixed rank $t$ the map $g\mapsto ga_t$ is a bijection of $G$, so
each men-side rank occurs exactly once in every row and column. The women's side is identical.
\end{proof}

We can now say precisely what singles out the anti-phase templates among all
automorphism-maximal profiles. The following theorem is new, and it is the reason the
construction of Section~\ref{sec:templates} deserves to be called canonical.

\begin{theorem}[Anti-phase characterization]\label{thm:antiphase-char}
Let $P(G,A,B)$ be an automorphism-maximal profile in regular normal form. Then
\[
  \rM(m,w)+\rW(w,m)=n+1 \ \text{ for all } (m,w)
  \quad\Longleftrightarrow\quad
  b_s=a_{n-1-s}^{-1}\ \text{ for all } s .
\]
That is, among all $(n!)^2$ regular normal forms over a fixed group $G$, exactly the $n!$
choices with $B=(a_{n-1}^{-1},a_{n-2}^{-1},\dots,a_0^{-1})$ carry the constant rank-sum
identity; these are precisely the regular anti-phase templates $P(G,A)$ of
Definition~\ref{def:template}.
\end{theorem}

\begin{proof}
In $P(G,A,B)$ we have $\rM(m_g,w_{ga_t})=t+1$ and $\rW(w_{hb_s},m_h)=s+1$; equivalently
$\rW(w_h,m_{hb_s})$... let us compute directly. Fix a pair $(m_g,w)$ and write $w=w_{ga_t}$,
which determines $t$ uniquely; then $\rM(m_g,w)=t+1$. Put $h=ga_t$. The woman $w_h$ ranks
$m_{hb_s}$ at rank $s+1$, so $\rW(w_h,m_g)=s+1$ where $s$ is the unique index with
$hb_s=g$, i.e.\ with
\[
  b_s = h^{-1}g = (ga_t)^{-1}g = a_t^{-1}.
\]
Thus $s=s(t)$ is determined by $b_{s(t)}=a_t^{-1}$, and
\[
  \rM(m_g,w)+\rW(w,m_g) = (t+1)+(s(t)+1) = t+s(t)+2 .
\]
Crucially this does not depend on $g$. Hence the rank sum is constantly $n+1$ if and only if
$t+s(t)+2=n+1$, i.e.\ $s(t)=n-1-t$, for every $t$; and by definition of $s(t)$ this says
exactly $b_{n-1-t}=a_t^{-1}$ for every $t$, which is the asserted condition after
re-indexing.
\end{proof}

\begin{remark}
Theorem~\ref{thm:antiphase-char} shows that the two structural features used throughout this
paper are independent coordinates: \emph{automorphism-maximality} pins down the normal form
$P(G,A,B)$ (Theorem~\ref{thm:normalform}) and already forces rank-Latinness
(Corollary~\ref{cor:autmax-latin}), while the \emph{constant rank sum} pins down $B$ in terms
of $A$. The anti-phase templates are exactly the intersection of the two conditions.
\end{remark}

\section{Regular anti-phase templates}\label{sec:templates}

\begin{definition}[Regular anti-phase template]\label{def:template}
Let $G$ be a finite group of order $n$ and $A=(a_0,\dots,a_{n-1})$ an ordering of its elements.
Define $P(G,A)\ :=\ P(G,A,B)$ with $b_s=a_{n-1-s}^{-1}$, i.e.
\[
  m_g:\; w_{ga_0}\succ w_{ga_1}\succ\cdots\succ w_{ga_{n-1}},
  \qquad
  w_h:\; m_{ha_{n-1}^{-1}}\succ m_{ha_{n-2}^{-1}}\succ\cdots\succ m_{ha_0^{-1}} .
\]
For $t=0,\dots,n-1$, the \emph{canonical matching} $\mu_t$ is $\mu_t(m_g)=w_{ga_t}$. The
\emph{adjacent quotients} of $A$ are
\[
  q_b \;=\; a_{b-1}a_b^{-1}, \qquad b=1,\dots,n-1 .
\]
\end{definition}

Note $q_b\ne e$ for every $b$, since $a_{b-1}\ne a_b$.

\begin{proposition}[Latin and anti-phase identities]\label{prop:tmpl-latin}
For every $G$ and $A$, the profile $P(G,A)$ is rank-Latin and satisfies
$\rM(m,w)+\rW(w,m)=n+1$ for every pair; consequently $\EM+\EW=n(n+1)$ for every perfect
matching.
\end{proposition}

\begin{proof}
Rank-Latinness is Corollary~\ref{cor:autmax-latin}, and the rank-sum identity is the
``$\Leftarrow$'' direction of Theorem~\ref{thm:antiphase-char}. Summing over the $n$ matched
pairs gives the egalitarian identity.
\end{proof}

Throughout the rest of this section we use the following coordinates. Since $A$ lists every
element of $G$ exactly once, any matching $\mu$ of $P(G,A)$ is encoded by its \emph{index
function}
\[
  d:G\to\{0,\dots,n-1\},\qquad \mu(m_g)=w_{g a_{d_g}},
\]
and $\mu$ is a matching precisely when $g\mapsto g a_{d_g}$ is a bijection of $G$. In these
coordinates $\rM(m_g,\mu(m_g))=d_g+1$, and the woman $w_{ga_e}$ ranks $m_g$ at rank $n-e$.

\begin{lemma}[Blocking in coordinates]\label{lem:blocking}
Let $\mu$ have index function $d$, and for $h\in G$ let $t_h$ denote the index of $h$'s
partner, i.e.\ $t_h=d_x$ where $x$ is the unique element with $xa_{d_x}=h$. Then $\mu$ is
stable if and only if
\[
  \text{for all } g\in G \text{ and all } e\in\{0,\dots,n-1\}:\qquad
  e<d_g \ \Longrightarrow\ t_{ga_e}\ \ge\ e .
\]
\end{lemma}

\begin{proof}
A pair $(m_g,w_h)$ blocks iff $m_g$ prefers $w_h$ to $\mu(m_g)$ and $w_h$ prefers $m_g$ to her
partner. Writing $h=ga_e$, the first condition says $e+1<d_g+1$, i.e.\ $e<d_g$. For the second,
$w_h$ ranks $m_g$ at rank $n-e$ and ranks her partner at rank $n-t_h$; she prefers $m_g$ iff
$n-e<n-t_h$, i.e.\ iff $t_h<e$. So a blocking pair exists iff there are $g,e$ with $e<d_g$ and
$t_{ga_e}<e$; the stated condition is the negation.
\end{proof}

\begin{proposition}[Canonical stability]\label{prop:canon-stable}
For every $G$, every ordering $A$ and every $t$, the canonical matching $\mu_t$ is stable in
$P(G,A)$.
\end{proposition}

\begin{proof}
Here $d\equiv t$ is constant, hence $t_h=t$ for every $h$. The implication of
Lemma~\ref{lem:blocking} requires, for $e<d_g=t$, that $t_{ga_e}=t\ge e$, which holds.
\end{proof}

The next lemma is the engine of everything that follows. It says that the top level set of the
index function of any stable matching is not merely nonempty but \emph{group-theoretically
rigid}.

\begin{lemma}[Coset lemma]\label{lem:coset}
Let $\mu$ be a stable matching of $P(G,A)$ with index function $d$, let $b=\max_{g\in G} d_g$,
and suppose $b\ge 1$. Put $D_b=\{g\in G : d_g=b\}$ and $H_b=\langle q_b\rangle$. Then
\[
  g\in D_b \ \Longrightarrow\ gq_b\in D_b ,
\]
so $D_b$ is a union of left cosets of $H_b$. In particular $|H_b|$ divides $|D_b|$.
\end{lemma}

\begin{proof}
Let $g\in D_b$, so $d_g=b\ge 1$, and consider the woman $w_{ga_{b-1}}$. Since $b-1<b=d_g$, man
$m_g$ prefers her to his own partner $w_{ga_b}$. Let $m_x$ be her partner, so
\begin{equation}\label{eq:xrel}
  x\,a_{d_x} \;=\; g\,a_{b-1}.
\end{equation}
She ranks $m_g$ at rank $n-(b-1)=n+1-b$ and ranks $m_x$ at rank $n-d_x$. Stability forbids the
blocking pair $(m_g,w_{ga_{b-1}})$, so she must strictly prefer $m_x$; note $m_x\ne m_g$
because $d_g=b\ne b-1$. Hence $n-d_x<n+1-b$, i.e.\ $d_x>b-1$, i.e.\ $d_x\ge b$. Maximality of
$b$ gives $d_x=b$, so $x\in D_b$. Substituting $d_x=b$ into~\eqref{eq:xrel} gives
$x a_b=g a_{b-1}$, whence
\[
  x \;=\; g\,a_{b-1}a_b^{-1} \;=\; g\,q_b .
\]
Thus $gq_b=x\in D_b$, proving the implication. Iterating, $D_b$ is closed under right
multiplication by every power of $q_b$, hence $gH_b\subseteq D_b$ for every $g\in D_b$; so
$D_b$ is a union of left cosets of $H_b$, and $|H_b|$ divides $|D_b|$.
\end{proof}

\begin{theorem}[Sharpened generator criterion]\label{thm:gencrit}
Let $\mu$ be a stable matching of $P(G,A)$ with index function $d$ and maximal index
$b=\max_g d_g$. If $b=0$, or if $b\ge1$ and $\langle q_b\rangle=G$, then $\mu=\mu_b$.

Consequently, if $\langle q_b\rangle=G$ for every $b=1,\dots,n-1$, then
\[
  \Stab(P(G,A))=\{\mu_0,\dots,\mu_{n-1}\},
\]
so $|\Stab(P(G,A))|=n$ and the stable lattice is an $n$-element chain.
\end{theorem}

\begin{proof}
If $b=0$ then $d\equiv0$ and $\mu=\mu_0$. If $b\ge1$ and $\langle q_b\rangle=G$, then by
Lemma~\ref{lem:coset} the set $D_b$ is a nonempty union of left cosets of $G$ itself, hence
$D_b=G$; that is, $d\equiv b$ and $\mu=\mu_b$.

For the second statement, every stable $\mu$ has some maximal index $b$, and the hypothesis
covers every possible value of $b$, so $\mu$ is canonical; the canonical matchings are stable by
Proposition~\ref{prop:canon-stable}. Under $\mu_k$ every man receives rank $k+1$, so the
men-dominance order on $\{\mu_0,\dots,\mu_{n-1}\}$ is the chain
$\mu_0\succ_M\cdots\succ_M\mu_{n-1}$.
\end{proof}

Note that the hypothesis is used only at the single value $b=\max_g d_g$ attached to $\mu$; this
is what ``sharpened'' refers to, and it is exactly the leverage needed for the converse.

\subsection{The converse}

We now show that the generator criterion is not merely sufficient but \emph{necessary}, by
constructing non-canonical stable matchings whenever some adjacent quotient fails to generate.
The construction is completely explicit.

\begin{theorem}[Converse of the generator criterion]\label{thm:converse}
Let $P(G,A)$ be a regular anti-phase template, let $b\in\{1,\dots,n-1\}$, and set
$H_b=\langle q_b\rangle$. Let $K\subseteq G$ be any union of left cosets of $H_b$, and define
\[
  d^{(b,K)}_g \;=\;
  \begin{cases}
    b, & g\in K,\\
    b-1, & g\notin K.
  \end{cases}
\]
Then $d^{(b,K)}$ is the index function of a stable matching $\mu^{(b,K)}$ of $P(G,A)$. Moreover
$\mu^{(b,G)}=\mu_b$ and $\mu^{(b,\emptyset)}=\mu_{b-1}$, while for every $K$ with
$\emptyset\subsetneq K\subsetneq G$ the matching $\mu^{(b,K)}$ is \emph{not} canonical.

Consequently
\[
  |\Stab(P(G,A))|=n
  \qquad\Longleftrightarrow\qquad
  \langle q_b\rangle=G \ \text{ for every } b=1,\dots,n-1 .
\]
\end{theorem}

\begin{proof}
\emph{$\mu^{(b,K)}$ is a matching.} We must check that $\varphi(g)=g\,a_{d_g}$ is a bijection of
$G$, where $d=d^{(b,K)}$. We have $\varphi(K)=Ka_b$ and $\varphi(G\setminus K)=(G\setminus
K)a_{b-1}$. Since $q_b=a_{b-1}a_b^{-1}$, we have $a_b a_{b-1}^{-1}=q_b^{-1}\in H_b$, and since
$K$ is a union of left cosets of $H_b$ we get $Kq_b^{-1}=K$, i.e.
\[
  K a_b a_{b-1}^{-1}=K, \qquad\text{that is,}\qquad K a_b=K a_{b-1}.
\]
Hence $\varphi(K)=Ka_{b-1}$ and $\varphi(G\setminus K)=(G\setminus K)a_{b-1}$ are disjoint with
union $G a_{b-1}=G$. So $\varphi$ is a bijection.

\emph{$\mu^{(b,K)}$ is stable.} The index function takes only the two values $b-1$ and $b$. We
compute the partner indices $t_h$. If $h\in Ka_b$ then, by the displayed identity, its
$\varphi$-preimage lies in $K$ and has index $b$, so $t_h=b$; if $h\in(G\setminus K)a_{b-1}$
then its preimage lies in $G\setminus K$ and has index $b-1$, so $t_h=b-1$. In either case
\[
  t_h \;\ge\; b-1 \qquad\text{for every } h\in G .
\]
Now suppose $(m_g,w_h)$ were a blocking pair, with $h=ga_e$. By Lemma~\ref{lem:blocking} this
requires $e<d_g$ and $t_h<e$, hence
\[
  t_h \;<\; e \;<\; d_g \;\le\; b ,
\]
so $t_h<e\le b-1$ and therefore $t_h\le b-2$, contradicting $t_h\ge b-1$. Hence there is no
blocking pair and $\mu^{(b,K)}$ is stable.

\emph{Canonicity.} If $K=G$ then $d\equiv b$ and $\mu^{(b,K)}=\mu_b$; if $K=\emptyset$ then
$d\equiv b-1$ and $\mu^{(b,K)}=\mu_{b-1}$. If $\emptyset\subsetneq K\subsetneq G$ then $d$ takes
both values $b-1$ and $b$, so it is not constant and $\mu^{(b,K)}$ is not canonical.

\emph{The equivalence.} If every $q_b$ generates $G$, then $|\Stab(P(G,A))|=n$ by
Theorem~\ref{thm:gencrit}. Conversely, suppose $\langle q_{b}\rangle=H\subsetneq G$ for some
$b$. Then the index $[G:H]\ge2$, so there are $2^{[G:H]}\ge4$ unions of left cosets of $H$, and
hence at least $2^{[G:H]}-2\ge2$ choices of $K$ with $\emptyset\subsetneq K\subsetneq G$. Each
yields a distinct non-canonical stable matching (distinct $K$ give distinct index functions).
Together with the $n$ canonical matchings this gives $|\Stab(P(G,A))|\ge n+2>n$.
\end{proof}

\begin{theorem}[Counting bound]\label{thm:count}
For every finite group $G$ of order $n$ and every ordering $A$,
\[
  |\Stab(P(G,A))| \;\ge\; n \;+\; \sum_{b=1}^{n-1}\Bigl(2^{\,[G:\langle q_b\rangle]}-2\Bigr).
\]
\end{theorem}

\begin{proof}
The $n$ canonical matchings are stable (Proposition~\ref{prop:canon-stable}). Fix
$b\in\{1,\dots,n-1\}$ and $H_b=\langle q_b\rangle$. By Theorem~\ref{thm:converse}, each of the
$2^{[G:H_b]}-2$ sets $K$ with $\emptyset\subsetneq K\subsetneq G$ that are unions of left cosets
of $H_b$ yields a non-canonical stable matching $\mu^{(b,K)}$, and distinct $K$ yield distinct
matchings.

It remains to check that matchings arising from different values of $b$ are distinct. The index
function of a non-canonical $\mu^{(b,K)}$ attains exactly the two values $b-1$ and $b$. Two such
value sets $\{b-1,b\}$ and $\{b'-1,b'\}$ coincide only if $b=b'$. Hence the families indexed by
distinct $b$ are pairwise disjoint, and they are disjoint from the canonical family. Summing
gives the bound.
\end{proof}

\begin{remark}[Sharpness]\label{rem:sharp}
The bound of Theorem~\ref{thm:count} is an equality for every ordering $A$ of every group of
order at most $5$, and for both groups of order $4$ (Observation~\ref{obs:census}). It is not an
equality in general: for $G=\mathbb{Z}_6$ and $A=(0,2,4,1,3,5)$ the bound gives $20$ while
$|\Stab|=24$. Thus the two-level matchings $\mu^{(b,K)}$ of Theorem~\ref{thm:converse} do not
exhaust the stable set in general, and a complete count remains open
(Problem~\ref{prob:count}).
\end{remark}

\subsection{Consequences: primality, not cyclicity}

\begin{corollary}[Prime order]\label{cor:prime}
Let $|G|=p$ be prime, so $G\cong\mathbb{Z}_p$. Then for \emph{every} ordering $A$ of $G$,
\[
  \Stab(P(\mathbb{Z}_p,A))=\{\mu_0,\dots,\mu_{p-1}\},
\]
so $|\Stab|=p$ and the stable lattice is a $p$-element chain.
\end{corollary}

\begin{proof}
Each $q_b=a_{b-1}a_b^{-1}$ is a non-identity element of a group of prime order, hence generates
it. Apply Theorem~\ref{thm:gencrit}.
\end{proof}

\begin{corollary}[Cyclic chain-lattice extreme]\label{cor:cyclic-extreme}
Let $G=\mathbb{Z}_n$ and $A=(0,1,\dots,n-1)$. Then $P(G,A)=C_n$, every adjacent quotient equals
$-1$, which generates $\mathbb{Z}_n$, and therefore $|\Stab(C_n)|=n$ with an $n$-element chain
lattice.
\end{corollary}

\begin{proof}
In additive notation $q_b=a_{b-1}-a_b=(b-1)-b=-1$, a generator of $\mathbb{Z}_n$; apply
Theorem~\ref{thm:gencrit}. That $P(\mathbb{Z}_n,(0,1,\dots,n-1))=C_n$ is immediate from
Definition~\ref{def:template}: $m_i$ ranks $w_{i+t}$ at rank $t+1$, and $w_h$ ranks
$m_{h-a_{n-1-s}}$ at rank $s+1$, i.e.\ $w_j: m_{j+1}\succ m_{j+2}\succ\cdots\succ m_j$.
\end{proof}

The following example corrects a claim implicit in earlier versions of this work.

\begin{example}[Cyclicity is not enough]\label{ex:z4}
Let $G=\mathbb{Z}_4$ and $A=(0,2,1,3)$. The adjacent quotients are
\[
  q_1=0-2=2,\qquad q_2=2-1=1,\qquad q_3=1-3=-2=2 ,
\]
so $\langle q_1\rangle=\langle q_3\rangle=\{0,2\}$ has index $2$, while
$\langle q_2\rangle=\mathbb{Z}_4$. Theorem~\ref{thm:count} gives
$|\Stab|\ge 4+(2^2-2)+(2^1-2)+(2^2-2)=4+2+0+2=8$, and direct enumeration confirms
$|\Stab(P(\mathbb{Z}_4,(0,2,1,3)))|=8$. In particular a template over a \emph{cyclic} group can
have strictly more than $n$ stable matchings, and its lattice is not a chain. Similarly
$A=(0,1,3,2)$ gives $|\Stab|=6$.

Hence the chain-lattice phenomenon is not a consequence of cyclicity of $G$; by
Corollary~\ref{cor:prime} and Theorem~\ref{thm:converse}, it is governed by whether every
adjacent quotient generates, which for a fixed group is a property of the ordering $A$, and
which holds automatically for all $A$ exactly when $|G|$ is prime.
\end{example}

\begin{observation}[The Klein four-group template]\label{obs:klein}
Let $G=V_4=\mathbb{Z}_2\times\mathbb{Z}_2$ and
$A=\bigl((0,0),(1,0),(0,1),(1,1)\bigr)$. Every non-identity element of $V_4$ has order $2$, and
\[
  q_1=(1,0),\qquad q_2=(1,0)(0,1)=(1,1),\qquad q_3=(0,1)(1,1)=(1,0),
\]
so each $\langle q_b\rangle$ has order $2$ and index $2$. Theorem~\ref{thm:count} gives
\[
  |\Stab(P(V_4,A))| \;\ge\; 4+3\bigl(2^{2}-2\bigr) \;=\; 4+6 \;=\; 10 ,
\]
and direct enumeration confirms equality: $|\Stab(P(V_4,A))|=10$. Since $10=f(4)$ is the
classical maximum number of stable matchings over all size-$4$
profiles~\cite{Knuth,Gusfield,Karlin}, the Klein anti-phase template is a compact group-theoretic instance attaining that
maximum — and Theorem~\ref{thm:count} \emph{derives} the value $10$ rather than merely reporting
an enumeration.

The template is rank-Latin and automorphism-maximal with $|\Aut|=4$. Its stable lattice has $10$
elements, is not a chain (it has $3$ incomparable pairs), and its rotation poset has $6$
elements: it is the ordinal sum $A_2\oplus A_2\oplus A_2$ of three $2$-element antichains, whose
ideal lattice indeed has $4+3+3=10$ elements, consistent with $L\cong J(R)$.

Finally, in contrast with Theorem~\ref{thm:parity}(b) — which shows that the cyclic profile
$C_4$ has \emph{no} exactly sex-equal stable matching — the Klein template of the same size has
\emph{two}, namely the matchings with $\EM=\EW=10$. Thus, at $n=4$, replacing $\mathbb{Z}_4$ by
$V_4$ in the anti-phase construction restores exact sex-equality while preserving
rank-Latinness, the constant rank sum $\EM+\EW=n(n+1)=20$, and automorphism-maximality.
\end{observation}

\section{Independent composition and direct products}\label{sec:products}

\begin{definition}
Let $M=M_1\sqcup M_2$ and $W=W_1\sqcup W_2$ with $|M_1|=|W_1|$. Suppose every $m\in M_1$ prefers
every woman of $W_1$ to every woman of $W_2$, and every $w\in W_1$ prefers every man of $M_1$ to
every man of $M_2$. Let $S_i$ be the induced subprofile on $(M_i,W_i)$; write $S=S_1\otimes S_2$.
\end{definition}

\begin{proposition}\label{prop:product}
If $S=S_1\otimes S_2$ then $\Stab(S)\cong\Stab(S_1)\times\Stab(S_2)$,
$L(S)\cong L(S_1)\times L(S_2)$, and $R(S)\cong R(S_1)\sqcup R(S_2)$.
\end{proposition}

\begin{proof}
If some $m\in M_1$ is matched into $W_2$, then since $|M_1|=|W_1|$ some $w\in W_1$ is matched
into $M_2$; the separation hypothesis makes $(m,w)$ a blocking pair. So every stable matching
respects the two groups. Restriction gives stable matchings of $S_1$ and $S_2$, and conversely
the union of stable matchings of $S_1$ and $S_2$ has no blocking pair, within or across groups.
The men-dominance order is componentwise, giving the product lattice. The join-irreducibles of a
product of finite distributive lattices are those of either factor paired with the bottom
element of the other, so the rotation poset is the disjoint union.
\end{proof}

\begin{remark}
This formalizes a standard separability idea; it is used here only to fix a product convention
for decomposition questions.
\end{remark}

\section{Computational observations}\label{sec:comp}

All computations in this section are reproduced by the accompanying script
\texttt{anc/\allowbreak verify\_smp\_\allowbreak v15.py}, which uses only the Python standard
library.

\begin{observation}[Size-three census]\label{obs:census3}
Over all $(3!)^6=46656$ strict complete size-three profiles:
\begin{enumerate}
\item[(a)] $|\Stab(S)|\in\{1,2,3\}$, with distribution $34080:11484:1092$;
\item[(b)] the stable-lattice types are the three chains of lengths $1,2,3$;
\item[(c)] under $S_3\times S_3$ relabeling there are $1300$ isomorphism classes, and $669$
after also identifying sex-duality;
\item[(d)] among profiles with exactly three stable matchings there are $31$ label-isomorphism
classes, and $17$ after sex-duality;
\item[(e)] the maximum automorphism group order is $3$, consistent with
Proposition~\ref{prop:semireg}.
\end{enumerate}
\end{observation}

\begin{observation}[Census over orderings]\label{obs:census}
Fix a group $G$ of order $n$. Right-translating $A$ by $a_0^{-1}$ gives an isomorphic profile,
so we may normalize $a_0=e$ and range over the $(n-1)!$ remaining orderings. For every group of
order at most $6$ and every normalized ordering $A$ (and for sampled orderings of the five
groups of order $8$), the script verifies:
\begin{enumerate}
\item[(a)] the equivalence of Theorem~\ref{thm:converse}: $|\Stab(P(G,A))|=n$ if and only if
every $q_b$ generates $G$;
\item[(b)] the lower bound of Theorem~\ref{thm:count};
\item[(c)] the coset lemma (Lemma~\ref{lem:coset}) for every stable matching.
\end{enumerate}
The bound of Theorem~\ref{thm:count} is an equality for all $A$ when
$G\in\{\mathbb{Z}_2,\mathbb{Z}_3,\mathbb{Z}_4,V_4,\mathbb{Z}_5\}$ and, by
Corollary~\ref{cor:prime}, trivially so for $\mathbb{Z}_7$. It is strict for some orderings of
$\mathbb{Z}_6$ (e.g.\ $A=(0,2,4,1,3,5)$: bound $20$, true value $24$), of $S_3$, and of all five
groups of order $8$. Representative values for $n=4$, illustrating the dependence on $A$:
\smallskip

\begin{center}
\begin{tabular}{llcc}
\toprule
$G$ & $A$ (normalized) & $\bigl(|\langle q_1\rangle|,|\langle q_2\rangle|,|\langle q_3\rangle|\bigr)$ & $|\Stab|$ \\
\midrule
$\mathbb{Z}_4$ & $(0,1,2,3)=C_4$ & $(4,4,4)$ & $4$ \\
$\mathbb{Z}_4$ & $(0,3,2,1)$ & $(4,4,4)$ & $4$ \\
$\mathbb{Z}_4$ & $(0,1,3,2)$ & $(4,2,4)$ & $6$ \\
$\mathbb{Z}_4$ & $(0,3,1,2)$ & $(4,2,4)$ & $6$ \\
$\mathbb{Z}_4$ & $(0,2,1,3)$ & $(2,4,2)$ & $8$ \\
$\mathbb{Z}_4$ & $(0,2,3,1)$ & $(2,4,2)$ & $8$ \\
$V_4$ & any of the $6$ normalized orderings & $(2,2,2)$ & $10$ \\
\bottomrule
\end{tabular}
\end{center}
\smallskip

\noindent
In every row the value of $|\Stab|$ equals the bound of Theorem~\ref{thm:count}. The table also
makes Example~\ref{ex:z4} concrete: within the single cyclic group $\mathbb{Z}_4$, the number of
stable matchings ranges over $\{4,6,8\}$ as $A$ varies.
\end{observation}

\begin{observation}[Cyclic checks]
The script verifies Propositions~\ref{prop:cyclic-latin}, \ref{prop:shift} and
\ref{prop:spectrum}, Corollary~\ref{cor:chain} and Theorem~\ref{thm:parity} for
$n=2,\dots,8$, and $|\Aut(C_n)|=n$ for $n=2,\dots,6$. It also verifies
Theorem~\ref{thm:antiphase-char} exhaustively over all $(n!)^2$ regular normal forms
$P(G,A,B)$ for $G\in\{\mathbb{Z}_3,\mathbb{Z}_4,V_4\}$: the constant rank-sum profiles are
exactly the $n!$ anti-phase templates, with no exceptions.
\end{observation}

\section{Open problems}\label{sec:open}

Problems~3 and~4 of the previous version are settled by Theorems~\ref{thm:converse}
and~\ref{thm:count} respectively (the latter only in part: a bound, not a formula). The
following remain.

\begin{problem}\label{prob:count}
Determine $|\Stab(P(G,A))|$ exactly in terms of $G$ and the subgroup chain
$\langle q_1\rangle,\dots,\langle q_{n-1}\rangle$. By Theorem~\ref{thm:count} the answer is at
least $n+\sum_b (2^{[G:\langle q_b\rangle]}-2)$, with equality for all groups of order $\le5$
and both groups of order $4$, but not in general (Remark~\ref{rem:sharp}). Which stable
matchings are missed by the two-level family $\mu^{(b,K)}$?
\end{problem}

\begin{problem}
Describe the stable lattice $L(P(G,A))$ and the rotation poset $R(P(G,A))$ structurally. For
the Klein template, $R\cong A_2\oplus A_2\oplus A_2$ (Observation~\ref{obs:klein}); is there a
general formula in terms of the subgroups $\langle q_b\rangle$?
\end{problem}

\begin{problem}
Classify regular anti-phase templates $P(G,A)$ up to profile isomorphism. Right translation of
$A$ and application of an automorphism of $G$ give isomorphic templates; are these the only
identifications?
\end{problem}

\begin{problem}
Over all groups $G$ of order $n$ and all orderings $A$, how large can $|\Stab(P(G,A))|$ be, and
how does this compare with the maximum $f(n)$ over all size-$n$ profiles? At $n=4$ the Klein
template attains $f(4)=10$ (Observation~\ref{obs:klein}). Does the elementary abelian group
$\mathbb{Z}_2^k$, whose subgroups $\langle q_b\rangle$ all have index $2^{k-1}$, give
asymptotically good lower-bound constructions, in the spirit of~\cite{Benjamin,Karlin}?
\end{problem}

\begin{problem}
Relate the regular normal form of Theorem~\ref{thm:normalform} to autotopism groups of Latin
squares~\cite{Denes}.
\end{problem}

\begin{problem}
Extend the sex-equality analysis of Section~\ref{sec:cyclic} to general templates: for which
$(G,A)$ does $P(G,A)$ admit an exactly sex-equal stable matching? By
Proposition~\ref{prop:tmpl-latin} the egalitarian score is always $n(n+1)$, so exact
sex-equality means $\EM=n(n+1)/2$; Theorem~\ref{thm:parity} and Observation~\ref{obs:klein}
show the answer depends on more than $n$.
\end{problem}

\section*{Appendix: small cases}

For $n=2$ the stable matchings of $C_2$ are $\mu_0=(0,1)$ and $\mu_1=(1,0)$, with imbalances
$-2$ and $+2$. For $n=3$ they are $(0,1,2)$, $(1,2,0)$, $(2,0,1)$, with imbalances $-6,0,+6$;
the middle one is the unique exactly sex-equal matching, as predicted by
Theorem~\ref{thm:parity}(a). For $n=4$ the cyclic profile $C_4$ has the four shifts
$(0,1,2,3)$, $(1,2,3,0)$, $(2,3,0,1)$, $(3,0,1,2)$ with imbalances $-12,-4,+4,+12$ — no exactly
sex-equal matching, as predicted by Theorem~\ref{thm:parity}(b). By contrast the Klein template
at the same $n=4$ has ten stable matchings, two of which are exactly sex-equal
(Observation~\ref{obs:klein}), and the template $P(\mathbb{Z}_4,(0,2,1,3))$ has eight
(Example~\ref{ex:z4}).

\section*{Reproducibility}

The source package contains the ancillary script \texttt{anc/verify\_smp\_v15.py}, which uses
only the Python standard library and runs in well under a minute. It verifies: the stable shifts of
$C_n$ and the Latin, rank-energy and parity statements for $n=2,\dots,8$; $|\Aut(C_n)|=n$ for
$n=2,\dots,6$; the anti-phase characterization (Theorem~\ref{thm:antiphase-char}) exhaustively
for $\mathbb{Z}_3$, $\mathbb{Z}_4$ and $V_4$; the coset lemma, the converse and the counting
bound for all groups of order $\le6$ and sampled orderings of all five groups of order $8$; the
size-three census of Observation~\ref{obs:census3}; and the Klein-template data of
Observation~\ref{obs:klein}, including its rotation poset and its two sex-equal matchings.


\begin{thebibliography}{99}

\bibitem{Benjamin}
A.~T. Benjamin, C.~Converse, and H.~A. Krieger,
\emph{How do I marry thee? Let me count the ways},
Discrete Appl. Math. \textbf{59} (1995), 285--292.

\bibitem{Birkhoff}
G.~Birkhoff,
\emph{Lattice Theory}, 3rd ed.,
Amer. Math. Soc. Colloq. Publ. \textbf{25}, AMS, 1967.

\bibitem{Borodin}
M.~Borodin, E.~Chen, A.~Duncan, T.~Khovanova, B.~Litchev, J.~Liu, V.~Moroz, M.~Qian,
R.~Raghavan, G.~Rastogi, and M.~Voigt,
\emph{Sequences of the stable matching problem},
J. Integer Seq. \textbf{27} (2024); also arXiv:2201.00645.

\bibitem{Bubboloni}
D.~Bubboloni, M.~Gori, and C.~Meo,
\emph{Resolute and symmetric mechanisms for two-sided matching problems},
J. Math. Econom. \textbf{118} (2025), 103130.

\bibitem{Denes}
J.~D\'enes and A.~D. Keedwell,
\emph{Latin Squares and Their Applications},
Academic Press, 1974.

\bibitem{Faenza}
Y.~Faenza and X.~Zhang,
\emph{Legal assignments and fast EADAM with consent via classical theory of stable matchings},
preprint, 2018.

\bibitem{GaleShapley}
D.~Gale and L.~S. Shapley,
\emph{College admissions and the stability of marriage},
Amer. Math. Monthly \textbf{69} (1962), 9--15.

\bibitem{Gusfield}
D.~Gusfield and R.~W. Irving,
\emph{The Stable Marriage Problem: Structure and Algorithms},
MIT Press, 1989.

\bibitem{Irving}
R.~W. Irving, P.~Leather, and D.~Gusfield,
\emph{An efficient algorithm for the optimal stable marriage},
J. ACM \textbf{34} (1987), 532--543.

\bibitem{Karlin}
A.~R. Karlin, S.~Oveis Gharan, and R.~Weber,
\emph{A simply exponential upper bound on the maximum number of stable matchings},
in: Proc. 50th ACM STOC, 2018, pp.~920--925.

\bibitem{Kato}
A.~Kato,
\emph{Complexity of the sex-equal stable marriage problem},
Japan J. Indust. Appl. Math. \textbf{10} (1993), 1--19.

\bibitem{Knuth}
D.~E. Knuth,
\emph{Mariages stables},
Les Presses de l'Universit\'e de Montr\'eal, 1976.

\bibitem{Manlove}
D.~F. Manlove,
\emph{Algorithmics of Matching Under Preferences},
World Scientific, 2013.

\bibitem{McDermid}
E.~McDermid and R.~W. Irving,
\emph{Sex-equal stable matchings: complexity and exact algorithms},
Algorithmica \textbf{68} (2014), 545--570.

\end{thebibliography}
\end{document}